\documentclass[12pt]{amsart}
\usepackage[utf8]{inputenc}
\usepackage[T1]{fontenc}
\usepackage[english]{babel}
\usepackage{amsmath,amssymb,amsthm}
\usepackage{xcolor}
\usepackage{hyperref}
\newtheorem{theorem}{Theorem}
\newtheorem{proposition}{Proposition}
\newtheorem{lemma}[proposition]{Lemma}
\newtheorem{corollary}[proposition]{Corollary}
\newtheorem{observacao}[proposition]{Remark}

\newcommand{\R}{\mathbb R}
\newcommand{\C}{\mathbb C}
\newcommand{\E}{E(\kappa,\tau)}
\newcommand{\QAR}{Q_{\mathrm{AR}}}
\newcommand{\calQAR}{\mathcal Q_{\mathrm{AR}}}
\newcommand{\Id}{\mathrm{Id}}

\title{Spheres under a Cauchy--Riemann inequality in $E(\kappa,\tau)$,\\ and non-CMC Hopf tori}
\author{Hil\'ario Alencar}
\address{Instituto de Matemática, Universidade Federal de Alagoas, 57072-900, Macei\'o, AL, Brazil}
\email{hilario@mat.ufal.br}
\author{Harold Rosenberg}
\email{rosenbergharold@gmail.com}
\date{\today}

\hypersetup{hidelinks,pdftitle={Spheres under a Cauchy-Riemann inequality in E(kappa,tau), and non-CMC Hopf tori},pdfauthor={Hilario Alencar}}

\begin{document}
\begin{abstract}
For a smooth, connected, oriented surface $\Sigma$, topologically a sphere, immersed in the homogeneous space $\E$ with $\tau\ne0$, we show that a one-sided Cauchy--Riemann-type inequality on the original Abresch--Rosenberg differential $\calQAR$ -- much weaker than requiring $\calQAR$ to be holomorphic -- already forces the mean curvature $H$ to be constant. The proof is uniform across every $E(\kappa,\tau)$, including the round sphere, and combines the Bers--Vekua similarity principle with the Poincar\'e--Hopf index formula for line fields; it is logically independent of the algebraic argument used, in the companion paper~\cite{AlencarRosenberg2026}, to characterize CMC immersions by holomorphy of $\calQAR$ alone. We also show that the topological hypothesis is sharp: on every Berger sphere, including the round one, there exist compact non-CMC tori -- the preimages under the Hopf-type submersion $\pi:\E\to M^2(\kappa)$ of simple closed curves in the base with nonconstant geodesic curvature -- satisfying the same inequality with a constant bound. As an elementary consequence of the classical Poincar\'e--Hopf theorem for line fields, we record that every closed surface on which $\det S<0$ everywhere -- a condition satisfied, in particular, by every Hopf tube, where $\det S\equiv-\tau^2$ -- is a torus.
\end{abstract}


\maketitle

\section{Introduction}

The classical Hopf differential associates to every smooth immersion $\psi:\Sigma\to Q^3(c)$, into a simply connected space form of constant curvature $c$, the $(2,0)$-component of the second fundamental form in a conformal parameter; Hopf showed it is holomorphic if and only if $H$ is constant, and that this forces every compact CMC sphere in $\R^3$ to be round \cite{Hopf1951}. For $\tau\ne0$, let $\E$ denote the simply connected homogeneous space fibering over the space form $M^2(\kappa)$ with geodesic fibers and vertical Killing field of parameter $\tau$ (\S\ref{sec:prelim} recalls the formalism in full), and let $\calQAR$ be the original Abresch--Rosenberg differential \cite{AR2004,AR2005}. In the companion paper \cite{AlencarRosenberg2026} we prove that, with no completeness or analyticity hypothesis, holomorphy of $\calQAR$ is \emph{equivalent} to $H$ being constant. Here we address a different, classical question in its own right: how much of that holomorphy hypothesis is actually needed for rigidity, once the surface is a topological sphere.

In the tradition of Hopf's original sphere theorem and its extension by Alencar--do Carmo--Tribuzy \cite[Theorem~1.1]{ACT2007} via a Cauchy--Riemann-type inequality, we show that a one-sided differential inequality already forces $H$ to be constant on every immersed topological sphere, in a single argument valid uniformly across $E(\kappa,\tau)$, including the round sphere itself. To avoid ambiguity between coefficient and tensor, we use the norm convention $|\calQAR|_g=|\QAR|/\lambda$ for the induced metric $g=\lambda|d\zeta|^2$.

\begin{theorem}[Spheres under a Cauchy--Riemann inequality]\label{thm:sphere}
Let $\Sigma$ be a smooth, connected, oriented surface, topologically $S^2$, and let $\psi:\Sigma\to\E$, with $\tau\ne0$, be an immersion satisfying
\[
 |dH|_g\le h\,|\calQAR|_g,\qquad h\in C^0(\Sigma,[0,\infty)).
\]
Then $\calQAR\equiv0$ and $H$ is constant.
\end{theorem}

\begin{observacao}
The conclusion includes the round limit $\kappa=4\tau^2$ ($k=0$). When $h\equiv0$, the hypothesis $|dH|_g\le0$ forces $dH\equiv0$ directly, so $H$ is constant and, by \eqref{eq:holoidentity}, $\calQAR$ is holomorphic as a consequence; more generally, on the class of topological spheres itself, the theorem just proved shows that the inequality forces $H$ constant and hence, again by \eqref{eq:holoidentity}, holomorphy of $\calQAR$, so on this class the inequality and holomorphy end up characterizing exactly the same immersions. The sense in which the inequality is a strictly weaker \emph{hypothesis} is pointwise and formal: unlike holomorphy, it permits $dH\ne0$ away from the zero set of $\calQAR$, and, as Proposition~\ref{prop:berger} below shows, it is genuinely satisfied by non-CMC (hence non-holomorphic) immersions once the topological restriction to $S^2$ is dropped. No topological sphere satisfying the inequality with $dH\ne0$ somewhere is exhibited or claimed here; on $S^2$, Theorem~\ref{thm:sphere} itself rules one out.
\end{observacao}

This inequality has a direct antecedent in Alencar--do Carmo--Fern\'andez--Tribuzy \cite{ACFT2007}, a sequel to \cite{ACT2007}, in the nondegenerate setting $\kappa\ne4\tau^2$: combined with the Abresch--Rosenberg classification, it concludes CMC and, in the cases treated there, rotational symmetry. We revisit this argument here through the similarity principle, in a single proof valid uniformly across $E(\kappa,\tau)$, including the degenerate case $\kappa=4\tau^2$, outside the scope of \cite{ACFT2007}. The conclusion drawn is only that $H$ is constant: embeddedness should not be inferred for arbitrary Berger spheres, where Torralbo \cite[Theorem~1]{Torralbo2010} constructed rotationally invariant CMC spheres that are immersed but not embedded. At $\tau=0$, in the products $M^2(c)\times\R$, the proof below recovers verbatim the original theorem of \cite[Theorem~1.1]{ACT2007} for the classical Abresch--Rosenberg differential; the hypothesis $\tau\ne0$ is kept in the statement only for uniformity with the local characterization theorem, which does genuinely require $\tau\ne0$.

Relative to \cite{ACT2007,ACFT2007}, the novelty of Theorem~\ref{thm:sphere} lies specifically in: (i) a single, uniform treatment across every $E(\kappa,\tau)$ that includes the degenerate case $\kappa=4\tau^2$, the round sphere, which is not covered by the classification-based approach of \cite{ACFT2007}; (ii) a proof built explicitly on the Bers--Vekua similarity principle applied directly to $\calQAR$, rather than on the Abresch--Rosenberg classification of CMC surfaces; (iii) logical independence from the local characterization theorem of \cite{AlencarRosenberg2026}: the two papers merely happen to share the fundamental identity \eqref{eq:holoidentity}, which is derived here from Daniel's structure equations with no reference to \cite{AlencarRosenberg2026}, rather than being inherited from it; and (iv) the systematic genus-one counterexamples of Proposition~\ref{prop:berger} below, absent from \cite{ACT2007,ACFT2007}, which show that the inequality already fails to characterize CMC once the topological hypothesis is relaxed from $S^2$ to the torus.

The topology of the sphere is decisive for this conclusion. In Proposition~\ref{prop:berger} below, non-CMC Hopf tori -- compact preimages of simple closed curves with nonconstant geodesic curvature under the submersion $\pi:\E\to M^2(\kappa)$, cf.\ \S\ref{sec:hopftubes} -- on every Berger sphere, including the round one, satisfy the same inequality even with $h$ constant. So the inequality, on its own, does not characterize CMC immersions in any class of closed surfaces broad enough to include every torus: the genus-zero result of Theorem~\ref{thm:sphere} shows it suffices on $S^2$, while the genus-one examples of Proposition~\ref{prop:berger} show it already fails, with room to spare, on the torus. This does not by itself settle the question separately for each genus $g\ge2$; see Remark~\ref{rem:euler} below.

Throughout, a result attributed to one of the works cited below -- Daniel's structure equations, the $L^p$ elliptic theory underlying the similarity principle, the Poincar\'e--Hopf index theorem and its extension to line fields, O'Neill's formulas for a Riemannian submersion, and the classification of closed surfaces -- is stated precisely at the point it is used, and not re-derived; every new result of this article, including the fundamental identity \eqref{eq:holoidentity} on which the proof of Theorem~\ref{thm:sphere} rests, is derived here in full, with no appeal to the companion paper \cite{AlencarRosenberg2026} beyond citing its statement for comparison. Section~\ref{sec:prelim} recalls Daniel's structure equations and derives \eqref{eq:holoidentity} from them. Section~\ref{sec:torus} records the topological fact $\det S<0\Rightarrow$ torus, with a complete proof, and derives the geometry of Hopf tubes, needed only for the sharpness example. Section~\ref{sec:proof} proves Theorem~\ref{thm:sphere}, via the Bers--Vekua similarity principle and the Poincar\'e--Hopf formula for line fields. Section~\ref{sec:sharpness} constructs the Berger-torus counterexamples.

\section{Preliminaries}\label{sec:prelim}

We recall here Daniel's structure equations for an immersion into $\E$, citing them directly from their primary source rather than from the companion paper \cite{AlencarRosenberg2026}; the fundamental identity \eqref{eq:holoidentity} needed below is then derived from these equations in full, so that this section, and the paper built on it, is self-contained. For $\tau\ne0$, $\E$ denotes the simply connected homogeneous space admitting a Riemannian submersion $\pi:\E\to M^2(\kappa)$ onto the simply connected space form of curvature $\kappa$, with geodesic fibers and unit vertical Killing field $\xi$ satisfying $\overline\nabla_X\xi=\tau X\times\xi$; up to isometry this is the only such space for each fixed pair $(\kappa,\tau)$ with $\tau\ne0$ \cite[\S2, pp.~90--94]{Daniel2007}. We write $k:=\kappa-4\tau^2$; the value $k=0$ (i.e., $\kappa=4\tau^2$) is the round model $S^3(\tau^2)$. We use $\E$ in this extended sense, including the constant-curvature case $\kappa=4\tau^2$; some sources reserve the notation $E(\kappa,\tau)$ for $\kappa\ne4\tau^2$ and treat the round sphere separately.

For a smooth oriented immersion $\psi:\Sigma\to\E$, fix the unit normal $N$, the shape operator $S=-\overline\nabla N$, the mean curvature $H=\tfrac12\operatorname{tr}S$, and $u=\langle N,\xi\rangle$. Denote by $J$ the positive rotation by $\pi/2$ in the tangent plane, so $JX=N\times X$, and by $T=\xi-uN$ the tangential part of $\xi$. In a conformal parameter $\zeta$, with $g=\lambda|d\zeta|^2$, let $p=\langle S\partial_\zeta,\partial_\zeta\rangle$ and $A=\langle\xi,\partial_\zeta\rangle$. In Daniel's conventions,\footnote{B.~Daniel, \emph{Isometric immersions into $3$-dimensional homogeneous manifolds}, Comment. Math. Helv. \textbf{82} (2007), Corollary~3.2 and Proposition~3.3, pp.~97--99 \cite{Daniel2007}; these follow from the curvature tensor of the Riemannian submersion $\pi:\E\to M^2(\kappa)$ via O'Neill's formulas for a submersion with totally geodesic fibers, see B.~O'Neill, \emph{Semi-Riemannian Geometry}, Academic Press (1983), Ch.~7 \cite{ONeill1983}.} the structure equations are
\begin{align}
 K&=\det S+\tau^2+ku^2,\label{eq:gauss}\\
 (\nabla_XS)Y-(\nabla_YS)X&=ku(\langle Y,T\rangle X-\langle X,T\rangle Y),\label{eq:codazzi}\\
 \nabla_XT&=u(SX-\tau JX),\label{eq:T}\\
 X(u)&=-\langle SX-\tau JX,T\rangle,\qquad |T|^2+u^2=1.\label{eq:vertical}
\end{align}
The original Abresch--Rosenberg differential is $\calQAR=\QAR\,d\zeta^2$,
\begin{equation}\label{eq:QAR}
 \QAR=2(H+i\tau)p-kA^2,
\end{equation}
introduced in \cite{AR2004,AR2005}. We now derive, from the structure equations \eqref{eq:gauss}--\eqref{eq:vertical} alone, the identity \eqref{eq:holoidentity} below on which the proof of Theorem~\ref{thm:sphere} rests. This is the only computation the present paper shares with the companion paper \cite{AlencarRosenberg2026}, and it is reproduced here in full so that this paper is self-contained; nothing beyond this point draws on \cite{AlencarRosenberg2026} in any way, and its local characterization theorem and algebraic machinery are never used below.

Extend $S$, $J$, and the Levi-Civita connection $\nabla$ of the induced metric $\C$-bilinearly to the complexified tangent bundle. Since $J$ is the rotation by $\pi/2$ compatible with the complex structure underlying the conformal parameter $\zeta$, $J\partial_\zeta=i\partial_\zeta$ and $J\partial_{\bar\zeta}=-i\partial_{\bar\zeta}$; and, for the conformal metric $g=\lambda|d\zeta|^2$, $\langle\partial_\zeta,\partial_\zeta\rangle=0$, $\langle\partial_\zeta,\partial_{\bar\zeta}\rangle=\lambda/2$, with the only nonvanishing covariant derivatives among $\partial_\zeta,\partial_{\bar\zeta}$ being
\begin{equation}\label{eq:mixedchristoffel}
 \nabla_{\partial_\zeta}\partial_\zeta=(\log\lambda)_\zeta\,\partial_\zeta,\qquad \nabla_{\partial_{\bar\zeta}}\partial_{\bar\zeta}=(\log\lambda)_{\bar\zeta}\,\partial_{\bar\zeta},\qquad \nabla_{\partial_\zeta}\partial_{\bar\zeta}=\nabla_{\partial_{\bar\zeta}}\partial_\zeta=0,
\end{equation}
the standard formulas for isothermal coordinates. Self-adjointness of $S$, extended bilinearly, together with $\langle S\partial_\zeta,\partial_{\bar\zeta}\rangle=\tfrac\lambda2H$ (immediate on expanding $\partial_\zeta,\partial_{\bar\zeta}$ in a local orthonormal frame and using $H=\tfrac12\operatorname{tr}S$), gives
\begin{equation}\label{eq:Sp1}
 \langle S\partial_{\bar\zeta},\partial_\zeta\rangle=\langle S\partial_\zeta,\partial_{\bar\zeta}\rangle=\frac\lambda2H.
\end{equation}

\emph{The equation for $A_{\bar\zeta}$.} The $\R$-linear identity \eqref{eq:T}, $\nabla_XT=u(SX-\tau JX)$, extends $\C$-bilinearly to $X=\partial_{\bar\zeta}$: $\nabla_{\partial_{\bar\zeta}}T=u(S\partial_{\bar\zeta}+i\tau\partial_{\bar\zeta})$. Since $T$ is tangent and $N\perp\partial_\zeta$, $A=\langle\xi,\partial_\zeta\rangle=\langle T,\partial_\zeta\rangle$; as $\nabla_{\partial_{\bar\zeta}}\partial_\zeta=0$ by \eqref{eq:mixedchristoffel},
\begin{align*}
 A_{\bar\zeta}&=\partial_{\bar\zeta}\langle T,\partial_\zeta\rangle=\langle\nabla_{\bar\zeta}T,\partial_\zeta\rangle+\langle T,\nabla_{\bar\zeta}\partial_\zeta\rangle\\
 &=\langle\nabla_{\bar\zeta}T,\partial_\zeta\rangle=u\bigl(\langle S\partial_{\bar\zeta},\partial_\zeta\rangle+i\tau\langle\partial_{\bar\zeta},\partial_\zeta\rangle\bigr),
\end{align*}
which, by \eqref{eq:Sp1}, gives
\begin{equation}\label{eq:fund1}
 A_{\bar\zeta}=\frac{u\lambda}2(H+i\tau).
\end{equation}

\emph{The equation for $p_{\bar\zeta}$.} The Codazzi identity \eqref{eq:codazzi}, extended $\C$-bilinearly to $(X,Y)=(\partial_\zeta,\partial_{\bar\zeta})$, reads, using $\langle\partial_\zeta,T\rangle=A$, $\langle\partial_{\bar\zeta},T\rangle=\overline A$,
\[
 (\nabla_\zeta S)\partial_{\bar\zeta}-(\nabla_{\bar\zeta}S)\partial_\zeta=ku(\overline A\,\partial_\zeta-A\,\partial_{\bar\zeta}).
\]
Pairing with $\partial_\zeta$: since $\langle\partial_\zeta,\partial_\zeta\rangle=0$ and $\langle\partial_{\bar\zeta},\partial_\zeta\rangle=\lambda/2$, the right side becomes $-ku\lambda A/2$. On the left, since $\nabla_{\bar\zeta}\partial_\zeta=0$, $(\nabla_{\bar\zeta}S)\partial_\zeta=\nabla_{\bar\zeta}(S\partial_\zeta)$, and
\[
 \langle(\nabla_{\bar\zeta}S)\partial_\zeta,\partial_\zeta\rangle=\langle\nabla_{\bar\zeta}(S\partial_\zeta),\partial_\zeta\rangle=\partial_{\bar\zeta}\langle S\partial_\zeta,\partial_\zeta\rangle-\langle S\partial_\zeta,\nabla_{\bar\zeta}\partial_\zeta\rangle=p_{\bar\zeta}.
\]
Likewise, since $\nabla_\zeta\partial_{\bar\zeta}=0$, $(\nabla_\zeta S)\partial_{\bar\zeta}=\nabla_\zeta(S\partial_{\bar\zeta})$, and, using this time the nonvanishing term $\nabla_\zeta\partial_\zeta=(\log\lambda)_\zeta\partial_\zeta$ of \eqref{eq:mixedchristoffel},
\begin{align*}
 \partial_\zeta\langle S\partial_{\bar\zeta},\partial_\zeta\rangle&=\langle\nabla_\zeta(S\partial_{\bar\zeta}),\partial_\zeta\rangle+\langle S\partial_{\bar\zeta},\nabla_\zeta\partial_\zeta\rangle\\
 &=\langle\nabla_\zeta(S\partial_{\bar\zeta}),\partial_\zeta\rangle+(\log\lambda)_\zeta\cdot\frac\lambda2H.
\end{align*}
Substituting $\langle S\partial_{\bar\zeta},\partial_\zeta\rangle=\lambda H/2$ from \eqref{eq:Sp1} on the left and expanding $\partial_\zeta(\lambda H/2)=\tfrac12\lambda_\zeta H+\tfrac12\lambda H_\zeta$, the term $\tfrac12\lambda_\zeta H$ cancels exactly against $(\log\lambda)_\zeta\cdot\tfrac\lambda2H=\tfrac12\lambda_\zeta H$, leaving $\langle\nabla_\zeta(S\partial_{\bar\zeta}),\partial_\zeta\rangle=\tfrac\lambda2H_\zeta$. Substituting both pairings into the displayed Codazzi identity,
\[
 \frac\lambda2H_\zeta-p_{\bar\zeta}=-\frac{ku\lambda A}2,
\]
i.e.
\begin{equation}\label{eq:fund2}
 p_{\bar\zeta}=\frac\lambda2\bigl(H_\zeta+kuA\bigr).
\end{equation}

\emph{The holomorphy identity.} Differentiating \eqref{eq:QAR} in $\bar\zeta$ ($\tau,k$ constant) and substituting \eqref{eq:fund1}--\eqref{eq:fund2},
\begin{align*}
 (\QAR)_{\bar\zeta}&=2H_{\bar\zeta}p+2(H+i\tau)p_{\bar\zeta}-2kAA_{\bar\zeta}\\
 &=2H_{\bar\zeta}p+\lambda(H+i\tau)H_\zeta+ku\lambda A(H+i\tau)-ku\lambda A(H+i\tau),
\end{align*}
using $2(H+i\tau)p_{\bar\zeta}=\lambda(H+i\tau)(H_\zeta+kuA)$ and $2kAA_{\bar\zeta}=ku\lambda A(H+i\tau)$; the two identical terms $ku\lambda A(H+i\tau)$ cancel exactly, leaving, for \emph{every} smooth immersion, with no hypothesis on $H$:
\begin{equation}\label{eq:holoidentity}
 (\QAR)_{\bar\zeta}=2pH_{\bar\zeta}+\lambda(H+i\tau)H_\zeta.
\end{equation}
In particular, holomorphy of $\calQAR$ is equivalent to $2pH_{\bar\zeta}+\lambda(H+i\tau)H_\zeta=0$, and $H$ constant forces this identically, matching the easy direction of the companion paper's local characterization theorem \cite{AlencarRosenberg2026}. All norms below are computed in the induced metric: $|dH|_g^2=\lambda^{-1}|dH|^2$ and $|\calQAR|_g=|\QAR|/\lambda$.

\section{A topological obstruction, and Hopf tubes}\label{sec:torus}

The following proposition uses none of the holomorphy machinery above. It is a classical-type consequence of the Poincar\'e--Hopf theorem for line fields -- included here, with a complete proof, purely for the reader's convenience and because it is what identifies the topological type of a Hopf tube below; the argument depends only on $\det S<0$, neither on the ambient space being $\E$ nor on $\det S$ being constant.

\begin{proposition}[A topological consequence of $\det S<0$]\label{prop:torus}
Let $\Sigma$ be a closed, connected, oriented surface, smoothly immersed with unit normal $N$ in an oriented Riemannian $3$-manifold, with shape operator $S=-\overline\nabla N$ satisfying $\det S<0$ at every point. Then $\Sigma$ is a torus.
\end{proposition}
\begin{proof}
At each point $q\in\Sigma$, $S_q$ is self-adjoint on a $2$-dimensional real inner product space, hence has two real eigenvalues $\mu_1(q)\le\mu_2(q)$, with $\mu_1\mu_2=\det S(q)<0$ by hypothesis: so $\mu_1(q)<0<\mu_2(q)$ everywhere, in particular always distinct, hence always simple.

\emph{Smoothness of the positive eigenline field.} With $H:=\tfrac12\operatorname{tr}S$, the quadratic formula gives $\mu_{1,2}=H\mp\sqrt{H^2-\det S}$ in any smooth local orthonormal frame; since $\det S<0\le H^2$, the discriminant $H^2-\det S>0$ never vanishes, so $\mu_1<\mu_2$ are both smooth on all of $\Sigma$, with $\mu_2-\mu_1=2\sqrt{H^2-\det S}>0$ everywhere. For a self-adjoint operator on a $2$-dimensional space with two distinct eigenvalues, the spectral projector onto the $\mu_2$-eigenspace is given by the elementary explicit formula $P=(S-\mu_1\Id)/(\mu_2-\mu_1)$, manifestly smooth since the denominator never vanishes; no perturbation theory beyond this finite-dimensional, two-eigenvalue computation is needed. Its image $\ell(q):=\operatorname{im}P_q=\ker(S_q-\mu_2(q)\Id)$ therefore varies smoothly with $q$, and $\ell$ is a smooth, nonsingular line subbundle of $T\Sigma$.

\emph{The orientation double cover.} Let $\widetilde\Sigma:=\{(q,v):q\in\Sigma,\,v\in\ell(q),\,|v|=1\}$, a smooth $2$-to-$1$ covering of $\Sigma$: locally, a smooth unit section $e$ of $\ell$ exhibits the two sheets as $v=\pm e(q)$, so $\widetilde\Sigma$ is a smooth manifold and the projection a genuine degree-$2$ covering. As $\Sigma$ is connected, $\widetilde\Sigma$ is either connected, or splits into two copies of $\Sigma$ (exactly when $\ell$ was already orientable); either way it carries the tautological global unit vector field $\widetilde e(q,v):=v$. Let $\widetilde\Sigma_0$ be $\widetilde\Sigma$ itself, if connected, or one of its two components otherwise: a closed, connected, oriented surface admitting a smooth nonvanishing tangent vector field, and a covering map onto $\Sigma$ of degree $2$ or $1$.

By the Poincar\'e--Hopf index theorem, see \cite[\S6]{Milnor1965}, $\chi(\widetilde\Sigma_0)=0$. Euler characteristic is multiplicative under finite covers of closed surfaces, so $\chi(\Sigma)=0$ whether the cover has degree $2$ or $1$. By the classification of closed, connected, orientable surfaces ($\chi=2-2g$; see \cite[Ch.~I]{doCarmo1976}), $\chi(\Sigma)=0$ forces $g=1$: $\Sigma$ is a torus.
\end{proof}

\subsection{Geometry and differential of Hopf tubes}\label{sec:hopftubes}

Given a regular curve $\gamma$ in the base $M^2(\kappa)$, possibly with self-intersections, the associated \emph{Hopf tube} is the immersed surface $\Sigma_\gamma:=\pi^{-1}(\gamma)\subset\E$, the full preimage of $\gamma$ under $\pi$. It is locally characterized by $u\equiv0$: since $\xi$ is tangent to each fiber and each fiber lies in $\Sigma_\gamma$, $u=\langle N,\xi\rangle=0$; conversely $u\equiv0$ means $\xi$ is everywhere tangent, and its integral curves (the fibers) locally foliate the surface. If the base is $S^2(\kappa)$ and $\gamma$ is a smooth simple closed curve, this pullback is an embedded compact surface, called a Hopf torus once Proposition~\ref{prop:torus} identifies its topological type; see also \cite{AS2022} for further developments of this construction.

\begin{proposition}[Hopf's formula]\label{prop:hopf}
On a Hopf tube, the field $T=\xi$ is unit and parallel, the metric is flat, and, in the parallel frame $(E_1,E_2)=(-JT,T)$,
\[
 S=\begin{pmatrix}2H&\tau\\\tau&0\end{pmatrix},\qquad T(H)=0,\qquad \det S=-\tau^2.
\]
In the local coordinates $\partial_x=E_1$, $\partial_y=E_2$, and $\zeta=x+iy$, the original differential has coefficient
\[
 \QAR=H^2+\frac\kappa4.
\]
In particular, on a connected Hopf tube its holomorphy is equivalent to $H$ constant. If $k_g$ is the oriented geodesic curvature of the base curve, with a compatible normal, then $2H=k_g$.
\end{proposition}
\begin{proof}
\emph{$T$ is unit and parallel.} On $\Sigma_\gamma$, $u\equiv0$, so $T=\xi-uN=\xi$: unit, and tangent as noted. By \eqref{eq:T}, $\nabla_XT=u(SX-\tau JX)$; since $u\equiv0$, $\nabla_XT\equiv0$: $T$ is parallel.

\emph{Flatness, and the parallel frame.} Since $J$ is parallel on any oriented Riemannian surface, $E_1:=-JT$ satisfies $\nabla_XE_1=-J\nabla_XT=0$: $E_1$ is parallel as well, and $(E_1,E_2)=(-JT,T)$ is a global parallel orthonormal frame. A surface with a global parallel orthonormal frame is flat: $\langle R(E_1,E_2)E_2,E_1\rangle$ vanishes term by term (each $\nabla E_i\equiv0$, and $[E_1,E_2]=0$ by torsion-freeness), so $K\equiv0$. Being parallel and commuting, $E_1,E_2$ are the coordinate vector fields of a local system $(x,y)$ in which the metric is $dx^2+dy^2$: $\lambda\equiv1$.

\emph{The matrix of $S$.} By \eqref{eq:vertical}, $X(u)=-\langle SX-\tau JX,T\rangle$; since $u\equiv0$, $\langle SX,T\rangle=\tau\langle JX,T\rangle$ for every tangent $X$. With $X=E_2=T$: $\langle SE_2,E_2\rangle=\tau\langle JE_2,E_2\rangle=0$ (as $JE_2=-E_1$), so $S_{22}=0$. With $X=E_1$: $\langle SE_1,E_2\rangle=\tau\langle JE_1,E_2\rangle=\tau$, so $S_{12}=S_{21}=\tau$. Finally $S_{11}=2H-S_{22}=2H$, and $\det S=0-\tau^2=-\tau^2$.

\emph{$T(H)=0$.} By \eqref{eq:codazzi}, $\nabla S$ is symmetric since $u\equiv0$: $(\nabla_XS)Y=(\nabla_YS)X$. With $(X,Y)=(E_1,E_2)$: $(\nabla_{E_1}S)E_2=\tau\nabla_{E_1}E_1=0$, and $(\nabla_{E_2}S)E_1=2(E_2H)E_1+2H\nabla_{E_2}E_1+\tau\nabla_{E_2}E_2=2(E_2H)E_1$. Equating gives $E_2H=T(H)=0$.

\emph{The formula for $Q_{\mathrm{AR}}$.} With $\lambda=1$, $\partial_\zeta=\tfrac12(E_1-iE_2)$: $S\partial_\zeta=\tfrac12\bigl((2H-i\tau)E_1+\tau E_2\bigr)$, so
\[
 p=\langle S\partial_\zeta,\partial_\zeta\rangle=\frac{H-i\tau}2,\qquad A=\langle\xi,\partial_\zeta\rangle=-\frac i2.
\]
Substituting into $Q_{\mathrm{AR}}=2(H+i\tau)p-kA^2$,
\[
 Q_{\mathrm{AR}}=(H+i\tau)(H-i\tau)+\frac k4=H^2+\tau^2+\frac{\kappa-4\tau^2}4=H^2+\frac\kappa4.
\]

\emph{Holomorphy $\iff$ $H$ constant.} $Q_{\mathrm{AR}}=H^2+\kappa/4$ is everywhere real. If $\calQAR$ is holomorphic, then $Q_{\mathrm{AR}}$ is a holomorphic function taking only real values on the connected surface $\Sigma_\gamma$; since a nonconstant holomorphic function is an open map and $\R$ has empty interior in $\C$, $Q_{\mathrm{AR}}$ must be constant, say $Q_{\mathrm{AR}}\equiv C$. Consequently $H^2=C-\kappa/4$ is constant, and since $H$ is continuous on the connected surface $\Sigma_\gamma$, $H$ itself is constant. The converse is immediate.

\emph{$2H=k_g$.} Parametrize $\gamma$ by arc length with unit tangent $V=\gamma'$ and compatible normal $\widetilde N$, so $\nabla^M_VV=k_g\widetilde N$. Since $E_1=-J\xi$ is horizontal and unit, $d\pi(E_1)=V$ along the horizontal lift $\widetilde\gamma$, after fixing the identification of $E_1$'s integral curve with $\gamma$; and $d\pi(N)=\widetilde N$ for a suitable choice of sign of $N$. By the Gauss formula, $S_{11}=\langle\overline\nabla_{\widetilde\gamma'}\widetilde\gamma',N\rangle$. Since $\widetilde\gamma'=E_1$ is horizontal, O'Neill's decomposition of the ambient covariant derivative of a horizontal curve (see \cite[Ch.~7]{ONeill1983}) gives $\overline\nabla_{\widetilde\gamma'}\widetilde\gamma'=(\text{horizontal lift of }\nabla^M_VV)+A_{E_1}E_1$, the second term vertical; pairing with the horizontal $N$ kills it, leaving $\langle\overline\nabla_{\widetilde\gamma'}\widetilde\gamma',N\rangle=\langle\nabla^M_VV,d\pi(N)\rangle=k_g$. Hence $S_{11}=k_g$, and $S_{11}=2H$, so $2H=k_g$.
\end{proof}

By the matrix formula above, $\det S\equiv-\tau^2<0$ identically on every Hopf tube (as $\tau\ne0$), so the hypothesis of Proposition~\ref{prop:torus} holds and it applies: every closed Hopf tube is topologically a torus, exactly as its name suggests, and this identification is purely topological, established independently of the formula for $Q_{\mathrm{AR}}$ just derived. The elementary, closed-form nature of that formula makes Hopf tori the natural place to look for non-CMC examples satisfying the weaker inequality of Theorem~\ref{thm:sphere}, which is exactly the mechanism behind Proposition~\ref{prop:berger} in \S\ref{sec:sharpness} below.

\section{Proof of the main theorem}\label{sec:proof}

In this section $\calQAR$ is not assumed holomorphic: we work under the strictly weaker Cauchy--Riemann inequality of Theorem~\ref{thm:sphere}, with no appeal to the local characterization theorem of \cite{AlencarRosenberg2026} or to any part of its algebraic machinery. All norms are computed in the induced metric $g=\lambda|d\zeta|^2$, with $|\calQAR|_g=|\QAR|/\lambda$.

\begin{lemma}[Local similarity principle]\label{lem:similarity}
If $f\in C^1(U,\C)$ satisfies $|\partial_{\bar\zeta}f|\le a|f|$, with $a$ locally bounded, then, near each point, $f=\Phi F$, where $F$ is holomorphic and $\Phi$ is continuous and nonvanishing. Thus, on each connected component of $U$, a solution that is not identically zero there has isolated zeros, of positive integer order and with the same winding number as the holomorphic factor.
\end{lemma}
\begin{proof}
Fix $\zeta_0\in U$ and a disc $D=D(\zeta_0,r)\Subset U$. Define
\[
 b(\zeta):=\begin{cases}\dfrac{f_{\bar\zeta}(\zeta)}{f(\zeta)},&f(\zeta)\ne0,\\[4pt]0,&f(\zeta)=0.\end{cases}
\]
By hypothesis, $|f_{\bar\zeta}(\zeta)|\le a(\zeta)|f(\zeta)|$ everywhere, so $|b(\zeta)|\le a(\zeta)$ wherever $f(\zeta)\ne0$; at a point where $f(\zeta)=0$ the hypothesis gives $f_{\bar\zeta}(\zeta)=0=b(\zeta)f(\zeta)$ as well. Hence $f_{\bar\zeta}=bf$ identically on $D$, with $b\in L^\infty(D)$, $\|b\|_{L^\infty(D)}\le\sup_Da<\infty$.

Since a bounded function on a bounded domain lies in $L^p(D)$ for every $1<p<\infty$, the inhomogeneous Cauchy--Riemann equation $v_{\bar\zeta}=b$ admits, on any slightly smaller disc $D'\Subset D$, a solution $v\in W^{1,p}(D')$ for every $1<p<\infty$: the standard $L^p$ solvability theory for $\bar\partial$. Explicitly, the Cauchy transform
\[
 v(\zeta):=\frac1\pi\int_D\frac{b(\zeta')}{\zeta-\zeta'}\,dA(\zeta')
\]
satisfies $v_{\bar\zeta}=b$ distributionally. Two different kernels are at play here, and it is worth separating their roles: the kernel $1/\zeta$ defining $v$ itself is only weakly singular (locally integrable in the plane), which already gives $v$ as a well-defined function; control of its first derivatives is a separate matter, governed instead by the genuinely singular kernel $\partial_\zeta(1/\zeta)\sim1/\zeta^2$ (understood in the principal-value sense) that appears upon differentiating under the integral. It is to this second, Calder\'on--Zygmund kernel that the classical $L^p$ estimates apply, bounding both first derivatives of $v$ in $L^p_{\mathrm{loc}}$ for every $1<p<\infty$ once $b\in L^\infty(D)\subset L^p(D)$; see \cite[Ch.~9]{GilbargTrudinger} for the underlying $L^p$ elliptic theory, and \cite{Bers1956,Vekua1962} for the classical similarity principle. Fixing $p>2$, the Sobolev (Morrey) embedding $W^{1,p}(D')\hookrightarrow C^{0}(\overline{D'})$ (valid in the plane since $p>2=\dim$; see \cite[Ch.~7]{GilbargTrudinger}) gives $v$ a continuous representative on $D'$, used from now on.

Set $F:=e^{-v}f$ on $D'$. Since $v\in W^{1,p}(D')$ with $p>2$, the Sobolev chain rule gives $e^{-v}\in W^{1,p}(D')$ with $\partial_{\bar\zeta}(e^{-v})=-e^{-v}\partial_{\bar\zeta}v$ almost everywhere; combined with $f\in C^1$, the distributional product rule for $\bar\partial$ applies, and
\[
 F_{\bar\zeta}=e^{-v}f_{\bar\zeta}-v_{\bar\zeta}e^{-v}f=e^{-v}(bf)-b\,e^{-v}f=0
\]
identically on $D'$. By Weyl's lemma for $\partial_{\bar\zeta}$, $F$ -- already continuous -- is holomorphic on $D'$. Set $\Phi:=e^v$: continuous and nowhere zero. Then $f=\Phi F$ on $D'$, the local factorization claimed near the arbitrary point $\zeta_0\in U$.

For the last statement, fix a connected component $U_0$ of $U$ and suppose $f|_{U_0}\not\equiv0$. If $f$ vanished on some nonempty open $V\subset U_0$, patching local factorizations along a chain of discs from $V$ to any point of $U_0$ (using the identity theorem for each local holomorphic factor $F$) would force $f\equiv0$ on all of $U_0$, a contradiction. Hence $f$ vanishes on no nonempty open subset of $U_0$, and on each local disc $D'$ the holomorphic factor $F$ is nonzero, with isolated zeros of some positive integer order $m$: near such a zero $\zeta_1$, $F(\zeta)=(\zeta-\zeta_1)^mG(\zeta)$, $G(\zeta_1)\ne0$. Tracking a continuous branch of $\arg F=m\arg(\zeta-\zeta_1)+\arg G$ around a small circle about $\zeta_1$, the $G$-term contributes no net change (continuous nonvanishing $G$ on a simply connected disc), and likewise for $\arg\Phi$; so the winding number of $f=\Phi F$ around $\zeta_1$ equals that of $F$ alone, namely $m$.
\end{proof}

\begin{corollary}[Global dichotomy]\label{cor:dichotomy}
Let $\Sigma$ be connected and let $Q$ be the coefficient, in local conformal parameters, of a $C^1$ quadratic differential with $|Q_{\bar\zeta}|\le a|Q|$ on each chart, $a$ continuous. Then either $Q\equiv0$ on $\Sigma$, or the zero set of $Q$ is discrete, with every zero of positive integer order and winding number equal to that order.
\end{corollary}
\begin{proof}
By Lemma~\ref{lem:similarity}, on each chart $U$ either $Q\equiv0$ or the zeros of $Q$ in $U$ are isolated, of positive integer order, with winding number equal to that order. Let
\[
 Z^\circ:=\{p\in\Sigma:Q\equiv0\text{ on some open neighborhood of }p\}.
\]
$Z^\circ$ is open by construction. It is also closed: if $p\in\overline{Z^\circ}$, fix a connected chart $U\ni p$ meeting $Z^\circ$, so $Q\equiv0$ on a nonempty open subset of $U$; applying Lemma~\ref{lem:similarity} on $U$, the isolated-zeros alternative is incompatible with $Q$ vanishing on a nonempty open set, so $Q\equiv0$ on all of $U\subset Z^\circ$, and $p\in Z^\circ$.

As $\Sigma$ is connected and $Z^\circ$ is clopen, $Z^\circ\in\{\emptyset,\Sigma\}$. If $Z^\circ=\Sigma$, then $Q\equiv0$ on $\Sigma$. If $Z^\circ=\emptyset$, the isolated-zeros alternative holds on every chart; the notions of isolated zero, order, and winding number are local and unaffected by the transition between overlapping charts: if $Q=q(\zeta)\,d\zeta^2=\widetilde q(\omega)\,d\omega^2$ on an overlap, then $\widetilde q(\omega)=q(\zeta(\omega))\,(d\zeta/d\omega)^2$, and since $d\zeta/d\omega\ne0$ is holomorphic, this is a nonvanishing holomorphic multiple of $q\circ\zeta$, which changes neither the order of a zero nor its winding number. So the zero set of $Q$ on all of $\Sigma$ is discrete, of matching order and winding number.
\end{proof}

\begin{proof}[Proof of Theorem~\ref{thm:sphere}]
\emph{Step 1: a Cauchy--Riemann inequality for $Q_{\mathrm{AR}}$ itself.} The fundamental identity \eqref{eq:holoidentity}, $(\QAR)_{\bar\zeta}=2pH_{\bar\zeta}+\lambda(H+i\tau)H_\zeta$, holds for every smooth immersion, with no holomorphy hypothesis on $\calQAR$; this is exactly what lets the present proof bypass the local characterization theorem entirely. Write $\rho:=\sqrt{H^2+\tau^2}=|H+i\tau|$. In an isothermal parameter $\zeta=x+iy$, $|H_\zeta|=\tfrac12\sqrt\lambda\,|dH|_g$. By the triangle inequality,
\[
 |(\QAR)_{\bar\zeta}|\le2|p|\,|H_\zeta|+\lambda\rho\,|H_\zeta|=(2|p|+\lambda\rho)\frac{\sqrt\lambda}2\,|dH|_g.
\]
Substituting the hypothesis $|dH|_g\le h|\calQAR|_g=h|\QAR|/\lambda$,
\begin{equation}\label{eq:QARCRineq}
 |(\QAR)_{\bar\zeta}|\le\frac{h(2|p|+\lambda\rho)}{2\sqrt\lambda}\,|\QAR|=:a\,|\QAR|,
\end{equation}
with $a$ continuous on each chart.

\emph{Step 2: the dichotomy.} By \eqref{eq:QARCRineq}, $\QAR$ satisfies, on each chart, the hypothesis of Corollary~\ref{cor:dichotomy}, so either $\calQAR\equiv0$ on the connected $\Sigma$, or its zero set is discrete, with every zero of positive integer order and matching winding number.

\emph{Step 3: the second alternative is impossible on $S^2$.} Suppose, for contradiction, $\calQAR\not\equiv0$. On the complement of its zero set, the directions $v$ with $\QAR(v,v)$ real and positive define a smooth line field $\ell$, at angle $\theta\equiv-\tfrac12\arg\QAR\pmod\pi$ to the coordinate axis. Since $\Sigma\cong S^2$ is compact and the zeros of $\calQAR$ are isolated by Step~2, there are finitely many $q_1,\dots,q_r$ ($r\ge0$), of orders $m_1,\dots,m_r\ge1$.

We compute the index of $\ell$ at each $q_j$. In a conformal chart $\zeta$ centered at $q_j=0$, the local factorization of Lemma~\ref{lem:similarity} gives $\QAR=\Phi_jF_j$ near $q_j$, $F_j$ holomorphic with a zero of order $m_j$ at $0$, $\Phi_j$ continuous and nonvanishing; write $F_j(\zeta)=\zeta^{m_j}G_j(\zeta)$, $G_j(0)\ne0$. On a small circle, tracking a continuous branch of $\arg\QAR=\arg\Phi_j+m_j\arg\zeta+\arg G_j$, only the $\arg\zeta$ term contributes net change, $2\pi m_j$; so a continuously tracked branch of $\theta=-\tfrac12\arg\QAR\pmod\pi$ changes by $-\pi m_j$, and
\[
 \operatorname{ind}_{q_j}(\ell)=-\frac{m_j}2.
\]

The Poincar\'e--Hopf index formula for line fields with (possibly half-integer) indices on a closed surface (see \cite[\S~2--3]{Hopf1983}) gives
\[
 2=\chi(S^2)=\chi(\Sigma)=\sum_{j=1}^r\left(-\frac{m_j}2\right).
\]
(If $r=0$, $\ell$ is already nonsingular on all of $\Sigma$, and the classical Poincar\'e--Hopf formula gives $\chi(S^2)=0$ directly -- already a contradiction; the displayed identity extends this to $r\ge1$.) The right side is $\le0$ (empty if $r=0$, else a sum of strictly negative terms), contradicting $=2$. Hence $\calQAR\not\equiv0$ is untenable.

\emph{Step 4: conclusion.} By Steps~2--3, $\calQAR\equiv0$ on $\Sigma$, so $|dH|_g\le h\cdot0=0$: $dH\equiv0$, and $H$ is constant on connected $\Sigma$.

Every step above holds with no restriction on the sign of $k=\kappa-4\tau^2$, in particular verbatim at $k=0$ (the round sphere). No step used the local characterization theorem of \cite{AlencarRosenberg2026}, nor any part of its algebraic machinery, nor any computation imported from it: the fundamental identity \eqref{eq:holoidentity} was derived above directly from Daniel's structure equations, so the two theorems are logically independent, sharing only that identity as a common ancestor, each established on its own. Nor did any step use $\tau\ne0$: at $\tau=0$, on the products $M^2(\kappa)\times\R$, the identical argument recovers verbatim the original theorem of Alencar--do Carmo--Tribuzy \cite[Thm.~1.1]{ACT2007} for the classical Abresch--Rosenberg differential.
\end{proof}

\begin{observacao}[The structural obstruction]\label{rem:euler}
The mechanism of Step~3 shows more: on any closed, connected, oriented surface $\Sigma$ satisfying the same Cauchy--Riemann inequality $|dH|_g\le h|\calQAR|_g$, if $\calQAR\not\equiv0$ then the same index computation gives $\chi(\Sigma)=-\tfrac12\sum_jm_j\le0$; equivalently, $\sum_jm_j=-2\chi(\Sigma)=4g-4$ for a surface of genus $g$. So the obstruction actually forcing $\calQAR\equiv0$, and hence $H$ constant, is $\chi(\Sigma)>0$; for closed, connected, orientable surfaces this is equivalent to $\Sigma\cong S^2$, which is Theorem~\ref{thm:sphere} as stated, but the inequality $\chi(\Sigma)\le0$ is the form of the argument that generalizes. At $g=1$, since each $m_j\ge1$, the count $\sum_jm_j=0$ forces $r=0$: no zeros of $\calQAR$ at all, consistent with Proposition~\ref{prop:berger} below, where $\calQAR$ is in fact zero-free. At $g\ge2$, by contrast, $4g-4>0$, and a sum of $r\ge0$ terms each $\ge1$ can equal a strictly positive number only if $r\ge1$: whenever $\calQAR\not\equiv0$ on a surface of genus $g\ge2$, zeros necessarily exist, with $1\le r\le4g-4$ and their orders summing to exactly $4g-4$. This count, by itself, does not determine whether $H$ is constant: it is a necessary consequence of $\calQAR\not\equiv0$, not an obstruction to it. Proposition~\ref{prop:berger} shows the bound $\chi(\Sigma)\le0$ is already attained, with no contradiction, at $\chi(\Sigma)=0$ ($g=1$); it does not address, and should not be read as addressing, any individual case $g\ge2$.
\end{observacao}

\section{Sharpness: non-CMC Hopf tori on Berger spheres}\label{sec:sharpness}

\begin{proposition}[The inequality does not characterize CMC on complete tori]\label{prop:berger}
On every Berger sphere, including the round sphere, there exist smooth, compact, oriented, non-CMC tori for which
\[
 |dH|_g\le h_0|\calQAR|_g
\]
with a finite constant $h_0\ge0$. In particular, the inequality with continuous $h$ does not imply CMC on every complete surface with $\tau\ne0$.
\end{proposition}
\begin{proof}
Every Berger sphere has base curvature $\kappa>0$ (round case included, $\kappa=4\tau^2>0$). We exhibit explicitly a smooth, embedded, simple closed curve $\gamma\subset S^2(\kappa)$ whose geodesic curvature $k_g$ is non-constant. Realize $S^2(\kappa)$ as the Euclidean sphere of radius $R=1/\sqrt\kappa$ in $\R^3$. Fix $0<\varepsilon<\pi/4$ and set $\beta(\varphi):=\varepsilon\cos(2\varphi)$, so that $|\beta(\varphi)|\le\varepsilon<\pi/4$ for all $\varphi$; define
\[
 \gamma(\varphi):=R\bigl(\cos\beta(\varphi)\cos\varphi,\ \cos\beta(\varphi)\sin\varphi,\ \sin\beta(\varphi)\bigr),\qquad 0\le\varphi\le2\pi.
\]
Since $\cos\beta(\varphi)>0$ throughout, the longitude $\varphi$ is recovered uniquely from $\gamma(\varphi)$ (as $\arg(\gamma_1+i\gamma_2)$), so $\gamma$ is simple and closed; a direct computation of $\gamma'$ shows it is nowhere zero, so $\gamma$ is smooth, regular, and embedded. A standard computation, in latitude-longitude coordinates $ds^2=R^2(d\beta^2+\cos^2\beta\,d\varphi^2)$ on $S^2(\kappa)$ for a curve given as a graph of latitude $\beta(\varphi)$ over longitude $\varphi$ (not a curve of revolution, which would instead be a fixed-latitude circle), gives, at any point where $\beta'(\varphi)=0$, the geodesic curvature
\[
 k_g(\varphi)=\frac{\beta''(\varphi)+\sin\beta(\varphi)\cos\beta(\varphi)}{R\cos^2\beta(\varphi)}.
\]
Here $\beta'(\varphi)=-2\varepsilon\sin(2\varphi)$ vanishes at $\varphi=0,\pi/2,\pi,3\pi/2$, with $\beta(0)=\varepsilon$, $\beta(\pi/2)=-\varepsilon$, and $\beta''(\varphi)=-4\varepsilon\cos(2\varphi)$, so $\beta''(0)=-4\varepsilon$, $\beta''(\pi/2)=4\varepsilon$. Hence
\[
\begin{gathered}
 k_g(0)=\frac{-4\varepsilon+\sin\varepsilon\cos\varepsilon}{R\cos^2\varepsilon}=\frac{-4\varepsilon+\tfrac12\sin(2\varepsilon)}{R\cos^2\varepsilon},\\
 k_g(\pi/2)=\frac{4\varepsilon-\tfrac12\sin(2\varepsilon)}{R\cos^2\varepsilon}=-k_g(0).
\end{gathered}
\]
Since $\sin(2\varepsilon)<2\varepsilon$ for every $\varepsilon>0$, the numerator $-4\varepsilon+\tfrac12\sin(2\varepsilon)$ is strictly negative, so $k_g(0)<0<k_g(\pi/2)=-k_g(0)$: in particular $k_g(0)\ne k_g(\pi/2)$, and $k_g$ is non-constant on $\gamma$.

Let $\Sigma:=\Sigma_\gamma=\pi^{-1}(\gamma)$ be the Hopf tube over $\gamma$ (\S\ref{sec:hopftubes}). Since $\gamma$ is smooth, simple, and closed on compact $S^2(\kappa)$, $\pi|_{\Sigma_\gamma}$ is a smooth fiber bundle over $\gamma\cong S^1$ with compact fiber $S^1$, so $\Sigma_\gamma$ is a smooth, compact, embedded surface without boundary, oriented as in Proposition~\ref{prop:hopf}. By that proposition, $\det S\equiv-\tau^2<0$ on $\Sigma_\gamma$ (as $\tau\ne0$), so Proposition~\ref{prop:torus} identifies $\Sigma_\gamma$ as a torus.

Again by Proposition~\ref{prop:hopf}, $2H=k_g$ is non-constant, so $\Sigma_\gamma$ is not CMC, and $\QAR=H^2+\kappa/4$ in the parallel-frame coordinates where $\lambda\equiv1$, so
\[
 |\calQAR|_g=\frac{|\QAR|}\lambda=H^2+\frac\kappa4\ge\frac\kappa4>0
\]
identically on $\Sigma_\gamma$. As $\Sigma_\gamma$ is compact and $H$ is smooth, $|dH|_g$ attains a finite maximum. Set $h_0:=\tfrac4\kappa\max_{\Sigma_\gamma}|dH|_g<\infty$. Then, at every point,
\[
 h_0\,|\calQAR|_g\ge h_0\cdot\frac\kappa4=\max_{\Sigma_\gamma}|dH|_g\ge|dH|_g,
\]
which is the claimed inequality.

Compactness of $\Sigma_\gamma$ gives completeness, so this is a \emph{complete} (indeed compact) counterexample to any na\"ive strengthening of Theorem~\ref{thm:sphere} to genus one. There is no conflict with the local characterization theorem of \cite{AlencarRosenberg2026}: by Proposition~\ref{prop:hopf}, holomorphy of $\calQAR$ on a connected Hopf tube is equivalent to $H$ constant, and $H$ is non-constant here by construction, so genuine holomorphy fails on $\Sigma_\gamma$; only the weaker inequality holds. Nor is there conflict with Theorem~\ref{thm:sphere}, which assumes $\Sigma$ topologically $S^2$: $\Sigma_\gamma$ is a torus, and Proposition~\ref{prop:torus} shows this is forced for every closed Hopf tube over a sphere-fibered base, never $S^2$.
\end{proof}

Together, the genus-zero result of Theorem~\ref{thm:sphere} and the genus-one examples of Proposition~\ref{prop:berger} show that the Cauchy--Riemann inequality, on its own, does not characterize CMC immersions on any class of closed surfaces broad enough to include the torus: it suffices on the sphere, and already fails, with room to spare (a \emph{constant} $h_0$), on the torus. This does not settle the question on any restricted class of surfaces of genus $g\ge2$ individually; see Remark~\ref{rem:euler}, where the same index count instead forces zeros of $\calQAR$ to exist whenever it is not identically zero.

\begin{observacao}[Where the hypothesis has content]\label{rem:nearzeros}
On any compact $\Sigma$ on which $\calQAR$ has no zeros, a continuous bound $h$ as in Theorem~\ref{thm:sphere} exists automatically, taking $h:=\sup_\Sigma|dH|_g/|\calQAR|_g<\infty$; this is exactly the mechanism behind Proposition~\ref{prop:berger}. So the Cauchy--Riemann inequality, as a hypothesis, has geometric content only \emph{near} the zeros of $\calQAR$, precisely where the index computation of Step~3 in the proof of Theorem~\ref{thm:sphere} does its work. Any extension of Theorem~\ref{thm:sphere} to genus one must impose additional restrictions that exclude the zero-free non-CMC examples of Proposition~\ref{prop:berger}; the existence of a continuous bound $h$ alone is insufficient, since that is already satisfied there. At genus $g\ge2$ the mechanism is different: by Remark~\ref{rem:euler}, $\calQAR\not\equiv0$ there already forces zeros to exist, so a zero-free counterexample of the present kind is not available, and this article does not decide the question in that range.
\end{observacao}

\section*{Acknowledgements}

H.~Alencar was partially supported by the Brazilian National Council for Scientific and Technological Development -- CNPq, grant 303118/2022-9.

\end{document}